\documentclass[a4paper,11pt]{article}
\usepackage{tikz}
\usepackage[makeroom]{cancel}
\usepackage[table]{xcolor}
\usepackage{amsthm}
\usepackage{amsmath}
\usepackage{amsfonts}
\usepackage{geometry}
\usepackage[utf8]{inputenc}
\usepackage{enumerate} 
\usepackage{subfig}
\usepackage{tikz}
\usepackage{tikz-3dplot}
\usepackage{color}
\usepackage{float}
\usepackage{leftidx}
\usepackage{bigints}
\usepackage[mathscr]{euscript}
\usepackage{amssymb,amsmath,dsfont,url,epsfig}
\usepackage{titlesec, blindtext, color}
\usepackage{graphicx, psfrag, epsf, times}
\definecolor{gray75}{gray}{0.75}

\newcommand{\sln}{\linespread{1}}
\newcommand*{\email}[1]{\href{mailto:#1}{\nolinkurl{#1}} } 
\titleformat{\chapter}[block]{\LARGE\bfseries\sln}{Chapter \thechapter}{11pt}{\newline\huge\bfseries}
\newtheorem{thm}{Theorem}[section]
\newtheorem{rem}{Remark}[section]

\newtheorem{defn}{Definition}[section]

\newtheorem{que}{Question}
\newtheorem{lem}{Lemma}[section]
\newtheorem{proposition}{Proposition}[section]

\begin{document}
\title{Finsler structure of the Apollonian weak metric on the unit disc}
\author  {Alok Kumar Pandey\footnote{E-mail: aalok3201@bhu.ac.in; Centre for Interdisciplinary Mathematical Sciences, Banaras Hindu University, Varanasi-221005, India.},  Ashok Kumar\footnote{E-mail: ashokkumar@hri.res.in; Harish-Chandra Research Institute, A CI of Homi Bhabha National Institute, Chhatnag Road, Jhunsi, Prayagraj-211019, India.}   and  Bankteshwar Tiwari\footnote{E-mail: btiwari@bhu.ac.in; Centre for Interdisciplinary Mathematical Sciences, Banaras Hindu University, Varanasi-221005, India}}
\date{} \maketitle

\begin{abstract}
\noindent 
In this paper, we {\it find} the Finsler structure of the Apollonian weak metric on the open unit disc in $\mathbb{R}^2$, which turns out to be a Randers type Finsler structure and we call it as Apollonian weak-Finsler structure. In fact the Apollonian weak-Finsler structure is the deformation of the hyperbolic Poincar\'e metric in the unit disc by a closed $1$-form. As a cosequence, the trajectories of the geodesic of this Apollonian weak-Finsler structure pointwise agrees with the geodesic of hyperbolic Poincar\'e metric in the open unit disc. Further, we explicitly calculate its $S$-curvature, Riemann curvature, Ricci curvature and flag curvature. It turns out that the $S$-curvature of 
the Apollonian weak-Finsler structure in the unit disc is bounded below by $\frac{3}{2}$, while its flag curvature $K$ satisfies $-\infty< K<-1$, in particular, it becomes a Hadamard manifold.
\end{abstract}
\section{Introduction}\label{sec1}
A weak metric on a set is a function which satisfies all axioms of a metric, except for the symmetry and separation axioms. The term weak metric was first introduced by Ribeiro \cite{RH} in 1943. Few well known examples of weak metrics are the Funk weak metric, the Apollonian weak metric and the Thurston metric. In \cite{APMT}, Papadopoulos and Troyanov introduce a weak metric, called the Apollonian weak metric on any subset of a Euclidean space. The arithmetic symmetrization of the Apollonian weak-metric is actually a   semi metric, called the Apollonian  semi metric, which  was introduced by Barbilian in 1934-35 and then re-discovered by Beardon \cite{BA} in 1995. It is defined for arbitrary domains in $\mathbb{R}^n$ and is M\"{o}bius invariant. 
In \cite{APMT}, Papadopoulos and Troyanov obtain the explicit formulas for the Apollonian weak metrics
in the upper-half plane and in the unit disk (see \cite[Theorem $1,2$]{APMT}. They have also pointed out that the Apollonian weak metrics are related to the conformal (Poincar\'e) model of the unit disk and the upper-half plane \cite{APMT}.

 \noindent Also, Papadopoulos and Troyanov observe that the isometry group of the Apollonian weak metric is quite different from the isometry group of the hyperbolic metric. On the other hand, they have shown that the hyperbolic lines are the geodesics in the unit disc for Apollonian weak metric \cite[Theorem 3]{APMT}. For a special class of plane domains, Beardon shows that the conformal Apollonian isometries are M\"{o}bius transformations. In \cite{ZI}, Ibragimov    shows that the Apollonian metric of a domain $D$ is either conformal at every point of $D$, at only one point of $D$ or at no point of $D$. The variational characterization of the Apollonian weak metric and Funk metric on the convex set in $\mathbb{R}^n$ has been investigated by Yamada \cite{SY}. \\
 
 \noindent A Finsler structure is defined on a differentiable manifold and consists of a family of Minkowski norms on the tangent spaces of the manifold (see Definition \ref{def 3.A01}). The importance of the Apollonian weak metric from a Finsler viewpoint is that it provides a natural and nontrivial example of a non-reversible Finsler structure arising  from classical conformal geometry and geometric function theory. A central Finsler-geometric question is: \textit{Given a distance-like function, does it arise from a Finsler structure?} For the Funk and Apollonian weak metrics, the answer is essentially yes in important domains such as the unit disk. This gives a concrete bridge between weak metrics, asymmetric distances, and smooth Finsler structures. This is conceptually important because many naturally occurring geometries are non-symmetric.
 Papadopoulos and Troyanov introduced tautological weak Finsler structure in a convex subset $ \Omega$ of  $\mathbb R^n$ and prove that the distance induced by it coincide with the  Funk weak metric on $\Omega$ (see \cite{HHG1}). In   \cite{SY}, variational formulas for the Finsler structure of the Funk and Apollonian weak metrices on the convex set in $\mathbb{R}^n$ has been obtained by Yamada. In \cite{AHB1}, various models of Funk-Finsler structure in Euclidean space with their curvature properties have been studied, where as in \cite{AHB1},  the explicit forms of Funk-Finsler structure in the hyperbolic plane geometry, which is of Randers type, with their curvature properties have been investigated.
 \\

 \noindent In this paper, we explicitly determine the Finsler structure associated with the Apollonian weak metric on the unit disc in $\mathbb R^2$, and show that it is of Randers type. It is interesting to note that the Finsler structure of the Funk weak metric on the unit disc is the deformation of the well known Klein-hyperbolic metric on the unit disc whereas the Finsler structure of the Apollonian weak metric metric on the unit disc is the deformation of the well known Poincar\'e-hyperbolic metrics on the unit disc, The study of the Finsler structures associated with the Apollonian weak metric on other convex domains in $\mathbb R^2$
  and, more generally, in $\mathbb R^n$, will be subject of the future study. 
 More precisely, in the present paper we obtain:
 
\begin{thm}\label{thm01}
\textnormal{    The Apollonian weak-Finsler structure of the Apollonian-weak metric in the unit disc $\mathbb{D}$ is a positive definite Randers structure given by}
\begin{equation*}
    \mathcal{F}_A(x,\xi)=\frac{|\xi|}{1-|x|^2}+\frac{\langle x,\xi\rangle}{1-|x|^2},
    \end{equation*}
   \textnormal{ here the first term on right hand side is the well known Poincar\'e metric $\frac{|\xi|}{1-|x|^2}:=\alpha(x,\xi)$ and the second term  $\frac{\langle x,\xi\rangle}{1-|x|^2}:=\beta(x,\xi)$ is a closed $1$-form given by $\beta=df$ with $f=-\frac{1}{2}\log(1-|x|^2)$}.
\end{thm}

 \begin{thm}\label{thm3.}\textnormal{The indicatrix $S_x$ of the Apollonian weak-Finsler structure  at the point $x$ in the unit disc $\mathbb D$ is an ellipse  with one of its foci at the point $x$ itself, origin as its center, major axis as the line joining origin to the point $x$, eccentricity $|x| $ and other focus at the point $-x$.}
\end{thm}
\noindent
  The various curvatures, for instance, $S$-curvature, flag curvatures  of the Apollonian weak-Finsler structure in the unit disc have also been computed explicitly (see Theorems \ref{thm4.1} and \ref{thm4.2}). As a consequence we have the following results on the bound of its $S$-curvature and flag curvature:
  \begin{thm}\label{thm1.3}
    \textnormal{The $S$-curvature of the Apollonian weak-Finsler structure $\mathcal{F}_A$ in the unit disc $\mathbb{D}$ is bounded below by $\frac{3}{2}$, i.e., $S \geq \frac{3}{2}\mathcal{F}_A$.}
\end{thm}
  \begin{thm}\label{thm1.4}
  \textnormal{The flag curvature $K$ of the Apollonian weak-Finsler structure $\mathcal{F}_A$ in the unit disc satisfies $-\infty < K<-1$. In particular, the unit disc equipped with Apollonian weak-Finsler structure  is a Hadamard manifold.}
\end{thm}

\noindent The paper is organized as follows.
 In Section \ref{sec2}, we discuss the preliminaries required for the paper. In Section \ref{sec3}, we explicitly construct the Apollonian weak-Finsler structure on the unit disc and obtain the indicatrix of the Apollonian weak-Finsler structure at any point in the unit disc.  Further, in Section \ref{sec4},  we investigate the geometry of the Apollonian weak-Finsler structure on the disc $\mathbb{D}$ and compute explicitly the spray coefficients, $S$-curvature,  Riemann curvature, Ricci curvature and the flag curvature.  Finally, in Section \ref{sec5}, we discuss the geometric realization of the Apollonian weak-Finsler structure on unit disc as pull back of a Randers structure in $\mathbb R^3$  on the upper sheet of hyperboloid of two sheets. We close the section by finding the Zermelo Navigation data for the Apollonian weak-Finsler structure on the unit disc $\mathbb{D}$.\\
 
\noindent In the sequel, we denote by $|.|$ and $\langle . \rangle$,  the Euclidean norm and the Euclidean inner product, respectively and $\mathbb{D}=\left\lbrace (x^1,x^2) \in \mathbb{R}^2 :  (x^1)^2+(x^2)^2 < 1 \right\rbrace$, the Euclidean disc centered at the origin with radius $1$ in $\mathbb{R}^2$. 
\section{Preliminaries}\label{sec2}
The theory of Finsler manifolds can be considered as a generalization of Riemannian manifolds, where the Riemannian metric is replaced by a so-called {\it Finsler} structure. A Finsler structure is a smoothly varying family of Minkowski norms in each tangent space of the manifold.\\
Let $ M $ be an $n$-dimensional smooth manifold and let $T_{x}M$ denote the tangent space of $M$ at $x$. The tangent bundle $TM$  of $M$ is the disjoint union of tangent spaces: $TM:= \sqcup _{x \in M}T_xM $. We denote the elements of $TM$ by $(x,\xi)$, where $\xi \in T_{x}M $ and $TM_0:=TM \setminus\left\lbrace 0\right\rbrace $.
 \begin{defn}[{Finsler structure \cite[\boldmath$\S 1.2$]{SSZ};  
  \cite[\boldmath$\S 16.2$]{Shiohama-BT}  }]\label{def 3.A01} 
  \textnormal{A {\it Finsler structure} on a smooth manifold $M$ is a continuous function $F:TM \to [0,\infty)$ satisfying the following conditions:}
   \begin{enumerate}[(i)]
       \item \textnormal{ $F$ is smooth on $TM_{0}$},
      \item \textnormal{ $F$ is a positively 1-homogeneous on the fibers of the tangent bundle $TM$,\\ i.e, $F(x,\lambda \xi)=\lambda F(x,\xi) ; \lambda >0$  and $(x,\xi)\in TM $}
 \item \textnormal{ The Hessian of $\displaystyle \frac{F^2}{2}$ with elements $\displaystyle g_{ij}=\frac{1}{2}\frac{\partial ^2F^2}{\partial \xi^i \partial \xi^j}$ is positive definite on $TM_0$.}\\
\textnormal{ The pair $(M, F)$ is called a Finsler space, and the quantities $g_{ij}$ are called components of the fundamental tensor of the Finsler structure $F$}.
   \end{enumerate}
  \end{defn}
 
 \noindent
 It is easy to see that Riemannian metrics are examples of Finsler structures. 
 \begin{defn}[{Finsler length of a curve  \cite[\boldmath$\S 1.3$]{SSZ}}]
    	\textnormal{The Finsler structure $F$ on the manifold $M$, induces a length structure $L_F$  on piecewise smooth curves in $M$. Let $\gamma:\left[ 0, 1\right] \rightarrow M$ be a piecewise smooth curve. Then the Finsler length of the curve $\gamma$ is defined by}
	\begin{equation}
		L_F\left(\gamma\right)=\int\limits_{0}^{1}F\left(\gamma\left(t\right),\dot \gamma\left(t\right)\right)dt. 
	\end{equation}
\end{defn}
\begin{defn}[{Finsler distance function \cite[\boldmath$\S 1.3$]{SSZ}}]
\textnormal{ The distance  function $d_F$ on $M$ induced by the Finsler structure $F$ is defined as:}
		\begin{equation}
			d_F(p,q)=\underset{\gamma}{\inf}L_F(\gamma),
		\end{equation}
		\textnormal{where  $p,q\in M$, and the infimum is taken over all piecewise smooth curves $\gamma$ joining $p$ to $q$ i.e.,  $p=\gamma(0)$ and $q=\gamma(1)$}. 
\end{defn} 
\begin{defn}[{Indicatrix in Finsler manifold \cite[\boldmath$\S 2.3 $]{OHTA}}]
   \textnormal{The indcatrix at a point $x$ of a Finsler manifold   $(M,F)$   is defined as}
     \begin{equation}
         S_x=\{\xi\in T_xM : {F}(\xi)=1\},
     \end{equation}
   \textnormal{i.e.,} $S_x=T_xM\cap F^{-1}(1)$. 
 \end{defn}
\begin{defn}[{Randers structure \cite[\boldmath$\S 1.2$]{SSZ}}] \textnormal{The Randers structure is the simplest non-Riemannian example of a Finsler structure.} \textnormal{Let $\displaystyle \alpha=\sqrt{a_{ij}(x)dx^idx^j}$ be a Riemannian metric  and $\beta=b_i(x)dx^i$ be a $1$-form on a smooth manifold  $M$ with $|\beta|_{\alpha}<1$, where $\displaystyle |\beta|_{\alpha} = \sqrt {a^{ij}(x)b_{i}(x)b_{j}(x)}$; \  then $F(x,\xi)=\alpha(x,\xi)+\beta(x,\xi)$, for all $x \in M, \xi \in T_xM$, is called a Randers structure on $M$}. \\

\end{defn}
 \begin{defn} [\text{The Riemann curvature tensor \cite[\boldmath$\S 4.1$]{CXSZ}}]\label{def2}  
 \textnormal{The Riemann curvature tensor $R={{R}_\xi:T_xM \rightarrow T_xM}$, for a Finsler space  $(M^n,F)$ is defined by
 ${R}_{\xi}(u)={R}^{i}_{k}(x,\xi)u^{k} \frac{\partial}{\partial x^i}$,\ $u=u^k\frac{\partial}{\partial x^k}$,\ where ${R}^{i}_{k}={R}^{i}_{k}(x,\xi)$ denote the coefficients of the Riemann curvature tensor of the Finsler structure $F$ and are given by,}
 	\begin{equation}\label{eqn2.2.10}
 			{R}^{i}_{k}=2\frac{\partial{G}^i}{\partial	x^k}-\xi^j\frac{\partial^2{G}^i}{\partial x^j\partial	\xi^k}+2{G}^j\frac{\partial^2 {G}^i}{\partial \xi^j\partial	\xi^k}-\frac{\partial{G}^i}{\partial \xi^j}\frac{\partial {G}^j}{\partial \xi^k}.
 	\end{equation}
 \textnormal{Here, $G^i=G^i(x,\xi)$ are local functions on $TM$, called the spray coefficients and defined by }
\begin{equation}\label{eqn2.1.14}
{G}^i=\frac{1}{4}{g}^{i{\ell}}\left\{\left[{F}^{2}\right]_{x^k\xi^{\ell}}\xi^k-	\left[{F}^{2}\right]_{x^{\ell}}\right\}.
\end{equation} 
 \end{defn}
 \noindent\textnormal{	The flag curvature $K=K(x,\xi, P)$, generalizes the sectional curvature in Riemannian geometry to the Finsler geometry and does not depend on whether one is using the Berwald, the Chern or the Cartan connection on the Finsler manifold.}

 	\begin{defn}[\text{Flag curvature \cite[\boldmath$\S 4.1$]{CXSZ}}] 
  \textnormal{	For a  plane $P\subset T_xM$ containing a non-zero vector $\xi$ called \textit{pole}, the \textit{flag curvature} $\textbf{K}(x,\xi,P)$ is defined by
 	\begin{equation}\label{eqn2.1.11}
 	\textbf{K}(x,\xi,P) :=\frac{g_\xi(R_\xi(u), u)}{g_\xi(\xi, \xi)g_\xi(u, u)- g_\xi(\xi, u)^2},
 	\end{equation}
 	where $u\in P$ is such that $P=\text{span} \left\{ \xi,u\right\} $.\\
    }
 		\end{defn}
\noindent The relation between the Riemann curvature  $R^i_j$ and the scalar flag curvature $\textbf{K}(x,\xi)$ of a Finsler structure $F$ is given by (see for more detail \cite[$\S 4.1$]{CXSZ})
 		 \begin{equation}\label{eqn2.1.12}
 		 R^i_j= \textbf{K}(x,\xi)\left\lbrace F^2 \delta^i_j-FF_{\xi^j}\xi^i\right\rbrace.
 		 \end{equation}
     
 \noindent It is well known that there is no canonical volume form on a Finsler manifold, like in the Riemannian case. Indeed, there are several well-known volume forms on the Finsler manifold, for instance,  the {\it Busemann-Hausdorff} volume form, the {\it Holmes-Thompson} volume form,  the {\it maximum } volume form,  the {\it minimum } volume form, etc. Here we discuss only the Busemann-Hausdorff volume form.
 \begin{defn}[{Busemann-Hausdorff Volume form of a Finsler manifold \cite[\boldmath$\S 2.2$]{SZ}; \cite[\boldmath$\S 2.7$]{BT}}]
\textnormal{ Let $(M,F)$ be an $n$-dimensional Finsler manifold and $(U,x^i)$ be a coordinate chart containing the point $x$. Let $\{\frac{\partial}{\partial x^i}\big |_x\}_{i=1}^n$ be the basis of $T_xM$ induced from the coordinate chart $(U,x^i)$. Then the \textit{Busemann-Hausdorff} volume form on the Finsler manifold $(M,F)$ is defined as:
  $dV_{BH}=\sigma_{BH}(x) \ dx$, where}
 \begin{equation}
 \sigma_{BH}(x)=\frac{\emph{Vol} ( B^n(1))}{\textnormal{Vol} \left\lbrace (\xi^i)\in \mathbb R^n: F(x,\xi^i\frac{\partial}{\partial x^i}\big |_x)<1\right\rbrace},
 \end{equation}
 \textnormal{and $dx=dx^1\wedge dx^2\wedge \dots \wedge dx^n$.
 Here $ B^n(1)$ denotes the Euclidean unit ball and \textnormal{Vol} denotes the canonical volume.}
 \end{defn}
\noindent
\textnormal{ The Busemann-Hausdorff volume form of the Randers structure can be explicitly given as follows:}
 \begin{lem}[{\cite[\boldmath$\S 1.3$]{SSZ}}]\label{lem 3.A1}
 \textnormal{The Busemann-Hausdorff volume form of the Randers structure $F =\alpha + \beta$ is given by,}
 \begin{equation}\label{eqn 3.A38}
 dV_{BH} = \left( 1-||\beta||^2_\alpha\right)^{\frac{n+1}{2}} dV_\alpha,
 \end{equation}
 \textnormal{where $dV_\alpha=\sqrt{\det (a_{ij})} \ dx$.}
 \end{lem}
 \noindent For the \textit{Busemann-Hausdorff} volume form 
  $dV_{BH}=\sigma_{BH}(x)dx$ on Finsler manifold $(M,F)$, the \textit{distortion} $\tau$ is defined by (see, \cite[$\S 5.1$]{SSZ})
\begin{equation*}
\tau (x,\xi) :=\log \frac{\sqrt{\det(g_{ij}(x,\xi))}}{\sigma _{BH}(x)}.
\end{equation*}
Now we define $S$-curvature of the Finsler manifold $(M, F)$ with respect to the volume form $dV_{BH}$.

\begin{defn}[{S-curvature, \cite[\boldmath$\S 5.1$]{SSZ}}]
 \textnormal{	For a vector $\xi\in T_xM\backslash \left\lbrace 0\right\rbrace$, let $\gamma=\gamma(t)$ be the geodesic with $\gamma(0)=x$ and $\dot{\gamma}(0)=\xi$.
 Then the $S$-curvature of the Finsler structure $F$ is defined by 	
\begin{equation*}
 S(x, \xi)=\frac{d}{dt}\left[ \tau \left(\gamma(t), \dot{\gamma}(t)\right) \right]_{|_{t=0}}.
 \end{equation*}}
 \end{defn}
\noindent The $S$-curvature of $F$ in  terms of spray coefficients $G^m$ is given by
 \begin{equation}\label{eqn2.1.15}
 S(x,\xi)=\frac{\partial G^m}{\partial \xi^m}-\xi^m\frac{\partial\left( \ln\sigma_{BH}\right) }{\partial x^m},
 \end{equation} 
 where $G^m$ are given by  \eqref{eqn2.1.14}.
         
    
    
    
     \begin{defn}[{Projectively flat space, \cite[\boldmath$\S 3.4$]{SSZ}}]\label{PF}
    \textnormal{ A Finsler structure $F = F(x,\xi)$ on an open subset $\mathcal{U}\in \mathbb{R}^n$ is said to be \textit{projectively flat} if all geodesics are straight in $\mathcal{U}$, that is, $\sigma(t)=f(t) a+ b$  for some constant vectors $a,b \in \mathbb{R}^n$ ($a\neq 0$) .  A Finsler structure $F$
     on a manifold $M$ is said to be locally projectively flat if at any point, there is a local coordinate system $(x^i)$ in which $F$ is projectively flat.}
 \end{defn}

 \begin {proposition} \cite[Proposition $3.4.8$]{SSZ}\label{ppn2.3}
\textnormal{A Randers structure $F=\alpha +\beta$ is locally projectively
 flat if and only if $\alpha$ is locally projectively flat and $\beta$ is closed.}
\end{proposition}
\begin{thm}\label{thm1}(\cite{DSSZ}, $\S 11.3$, \cite{SSZ}, $\S 3.4.8$)
 \textnormal{ If  $F=\alpha+\beta$  is a Randers structure on a manifold $M$ with $\beta$ a closed $1$-form, then the Finslerian geodesics have the same trajectories as the geodesics of the underlying Riemannian metric $\alpha$. Moreover, if  $(M, \alpha)$ has constant curvature, then $(M, F)$ is locally projectively flat and consequently, in this case 
 $(M, F)$  is  projectively equivalent to  $(M, \alpha)$.} 
 \end{thm}
\begin{defn}[{Weak metric, semi metric and metric \cite[\boldmath$\S 1$]{APMT}}]\label{AA}
\textnormal{ A weak metric on a set $X$ is a function $\delta: X\times X \rightarrow [0,\infty)$ satisfying}
 \begin{itemize}
     \item[(i)] $\delta(x, x) = 0$ \textnormal{ for all $x$ in $X$};
    \item[(ii)] $\delta(x, y) +\delta(y, z) \geq \delta(x, z)$ \textnormal{ for all $x, y$ and $z$ in $X$}.
    \end{itemize} 
    \textnormal{  A \textit{semi-metric} is a symmetric weak metric, that is, a weak metric satisfying}
    \begin{itemize}  
    \item[(iii)] $\delta(x, y) = \delta(y,x) $\textnormal{ \textnormal{for all $x$ and $y$ in $X$.}}
 \end{itemize}
  \textnormal{A metric is a symmetric weak metric satisfying $\delta(x,y)=0$ iff} $x=y$.
 \end{defn}
 \begin{defn}[{The Apollonian weak metric, \cite[\boldmath$\S 4$]{APMT}}]\label{AAA}
\textnormal{For any open subset $A\subset \mathbb{R}^n$ which is either bounded or whose boundary $\partial A$ is unbounded, the  Apollonian weak metric $\delta_A: A \times A \rightarrow \mathbb{R}$ is defined by}:\\
For $x,y \in A$
 \begin{equation}
     \displaystyle \delta_A(x,y)=\sup_{a\in \partial{A}}\log\left|\frac{x-a}{y-a}\right|,
 \end{equation}
 \textnormal{where $\partial{A}$ denotes boundary of the set $A$.}
 \end{defn}
\noindent  In the sequel, we restrict ourselves to the discussion on dimension $2$ and for computational simplicity, we equipped $\mathbb{R}^2$ with the complex structure. Therefore we write $\mathbb{R}^2$ and $\mathbb{C}$ interchangebly as per context. For instance the unit disc $\mathbb{D}$ centered at origin in $\mathbb{C}$ is given by 
\begin{equation*}
\mathbb{D}=\{z\in \mathbb{C}: |z|<1\}.
 \end{equation*}
\section{The Finsler structure of the Apollonian weak metric on the unit disc $\mathbb{D}$}\label{sec3}
In this section we show that the Apollonian weak metric on unit disc $\mathbb{D} $ is a Finsler structure and explicitly find the Apollonian weak-Finsler structure on unit disc $\mathbb{D}$. Further, we show that the indicatrix of the Apollonian weak-Finsler structure  at any point of disc $\mathbb{D}$ is an ellipse. 
We begin by considering the map $M: \mathbb{D}\times \mathbb{D} \rightarrow \mathbb{R}$ defined by
\begin{equation}\label{t1}
\displaystyle M(z_1,z_2)=\sup_{z \in \partial \mathbb{D}}\left|\frac{z_1-z}{z_2-z}\right|,
\end{equation}
where $z_1, z_2 \in \mathbb{D}$, and the Apollonian weak metric is defined by 
\begin{equation}
    \delta_A(z_1,z_2)=\log M(z_1,z_2).
\end{equation}
First, we  find the expression of the point $\zeta \in \partial \mathbb{D}$, where the supremum in \eqref{t1} is attained. Let us consider $\zeta =e^{it}$ and set
\begin{equation}\label{t2}
    f(t)=\left|\frac{z_1-e^{it}}{z_2-e^{it}}\right|^2.
\end{equation}
\begin{lem} \textnormal{ Let $z_1$ and $z_2$ be two distinct points lies in the open unit disc $\mathbb{D}$, then there is a clircle (a straight line or a circle) passing through $z_1$, $z_2 $ and their inverse points $\frac{z_1}{|z_1|^2},\frac{z_2}{|z_2|^2}$ with respect to the unit circle $|z|=1$, given by}
\begin{equation}\label{A11}
    (|z|^2+1)\rho_1-(\bar{\rho}z+\rho\bar{z})\rho_2=0.
\end{equation}
\textnormal{where $\rho_1=\bar{z_1}{z_2}-\bar{z_2}{z_1}$ and $\rho_2=z_2(1-z_1\bar{z_2})-z_1(1-\bar{z_1}z_2)$}. \textnormal{Moreover, the above clircle intersects the unit circle $|z|=1$ orthogonally.}
\textnormal{Further, if $\rho_1 \ne 0$ then \eqref{A11} represents a circle centered at $\rho$, given by}
\begin{equation}\label{t3}
    \rho=\frac{\rho_2}{\rho_1}=\frac{z_2(1-z_1\bar{z_2})-z_1(1-\bar{z_1}z_2)}{\bar{z_1}{z_2}-\bar{z_2}{z_1}},
    \end{equation}
\textnormal{and radius $R$  given by}
    \begin{equation}\label{t5}
         R=(|\rho|^2-1)^{1/2}.
     \end{equation}
\end{lem}
\begin{rem}
    \textnormal{The clircle represented by \eqref{A11} is actually the trajectories of the hyperbolic geodesic passing through $z_1$ and $z_2$, if the disc $\mathbb{D}$ is assumed to be equipped with Poinca\'re hyperbolic metric.}
\end{rem}
\begin{thm}\label{PPN1}
  \textnormal{  Let $z_1,z_2$ be any two distinct points in $\mathbb{D}$, then the supremum of $M$ in \eqref{t1} is attained at $a^{+}$ in $\partial\mathbb{D}$, where $a^{+}$ is the point of intersection of $\partial\mathbb{D}$ with the hyperbolic geodesic ray starting from point $z_1$ and passing through $z_2$.}
\end{thm}
\begin{proof}
\textnormal{In view of \eqref{t1} and \eqref{t2}, if the supremum  in \eqref{t1} is attained at the point $\zeta=e^{it}$, then $f'(t)=0$. From \eqref{t2} we have}
\begin{equation} \label{t20}
    f(t)=\frac{|z_1|^2+1-z_1 e^{-it}-\bar{z_1}e^{it}}{|z_2|^2+1-z_2 e^{-it}-\bar{z_2}e^{it}}.
\end{equation}
\textnormal{Differentaiting \eqref{t20} with respect to $t$ we have}
\begin{equation}\label{t11}
\begin{split}
    f'(t)=&\frac{\sin t[(1+|z_2|^2)(z_1+\bar{z_1})-(1+|z_1|^2)(z_2+\bar{z_2})]}{(|z_2|^2+1-z_2 e^{-it}-\bar{z_2}e^{it})^2} \\ & +\frac{i\cos t[(1+|z_2|^2)(z_2-\bar{z_2})-(1+|z_2|^2)(z_2-\bar{z_2})]+2i(z_2\bar{z_1}-z_1\bar{z_2})}{(|z_2|^2+1-z_2 e^{-it}-\bar{z_2}e^{it})^2}.
    \end{split}
\end{equation}
\newline
Solving, $f'(t)=0$. We obtain
\begin{equation}
 \begin{split}\label{t19}
    &\cos t[(1+|z_2|^2)(z_1-\bar{z_1})-(1+|z_1|^2)(z_2-\bar{z_2})]\\&-i\sin t[(1+|z_2|^2)(z_1+\bar{z_1})-(1+|z_1|^2)(z_2+\bar{z_2}) =2(z_1\bar{z_2}-z_2\bar{z_1}).
    \end{split}
\end{equation}
Here, we consider two cases:\\
\newline
\textbf{CASE(1): $z_1\bar{z_2}-z_2\bar{z_1}=0$ :}
Since $z_1$ and $z_2$ are distinct points, therefore, either $z_1\not=0$ or $z_2\not=0$. Without loss of generality, we can assume $z_2\not=0$, and then  $z_1\bar{z_2}-z_2\bar{z_1}=0$ is equivalent to $\frac{z_1}{z_2}=\frac{\bar{z_1}}{\bar{z_2}}$. Hence, $\frac{z_1}{z_2}$ is a real number.\\

\noindent Then the argument of $z_1$ and $z_2$ are either same or differ by an integral multiple of $\pi$. Let us first assume that argument of $z_1$and $z_2$ are same and assume that $z_1=r_1e^{i\theta}$ and  $z_2=r_2e^{i\theta}$  such that  $0<r_1 < r_2< 1$.\\
From \eqref{t19} we have,
\begin{equation}\label{t16}
\begin{split}
    & \cos t[(1+r_2^2)r_1(e^{i\theta}-e^{-i\theta})-(1+r_1^2)r_2(e^{i\theta}-e^{-i\theta})]\\
     &-i \sin t[(1+r_2^2)r_1(e^{i\theta}+e^{-i\theta})-(1+r_1^2)r_2(e^{i\theta}+e^{-i\theta})]=0,
\end{split}
\end{equation}

\begin{equation*}
\mbox{or}~~~~~    [\cos t(e^{i\theta}-e^{-i\theta})-i\sin t({e^{i\theta}+e^{-i\theta})}](r_1-r_2)(1-r_1r_2)=0.
\end{equation*}
After a small calculation one can see,
\begin{equation*}
    (e^{i(\theta-t)}-e^{-i(\theta-t)})[(r_1-r_2)(1-r_1r_2)]=0.
\end{equation*}
Since $r_1-r_2\not=0$ as ($1 > r_2>r_1>0$) and $1-r_1r_2\not=0$. Therefore,
\begin{equation*}
    (e^{i(\theta-t)}-e^{-i(\theta-t)})=0,~~\mbox{i.e.},~~ 2i \sin (t-\theta)=0,
\end{equation*}
Hence, either $t=\theta$ or $\theta+\pi$. Also one can see $f''(t)<0$ at $t=\theta$. Therefore $f$ attains  the maximum value at $t=\theta$, i.e., the supremum in equation \eqref{t1} is achieved at the point \textbf{$\zeta=a^{+}=e^{i\theta}$}. In this case $a^+$ is the point of intersection of the ray starting from $z_1$ and passing through $z_2$ with the $\partial \mathbb D$. Since $z_1$ and $z_2$ lie on the same diameter, the ray through $z_1$ and $z_2$ is actually the hyperbolic geodesic ray in the unit disc $\mathbb{D}$ equipped with the Poinca\'re hyperbolic metric and thus $a^+$ is actually the point of intersection of this geodesic ray  through $z_1$ and $z_2$ with $\partial \mathbb D$.\\

\noindent \textbf{CASE(2) : $z_2\bar{z_1}-z_1\bar{z_2}\not =0$ :} If the supremum of \eqref{t1} is attains at $\zeta=e^{it}$. Then by \eqref{t19} we have $f'(t)=0$ and obtain
\begin{equation}\label{t10}
\begin{split}
   &\cos t\left[\frac{(1+|z_2|^2(z_1-\bar{z_1})-(1+|z_1|^2)(z_2-\bar{z_2})}{z_1\bar{z_2}-z_2\bar{z_1}}\right]\\ &-i\sin t\left[\frac{(1+|z_2|^2(z_1+\bar{z_1})-(1+|z_1|^2)(z_2+\bar{z_2})}{z_1\bar{z_2}-z_2\bar{z_1}}\right]=2.  
\end{split}
\end{equation}
Consequently,
\begin{equation}\label{t4}
   (\rho+\bar{\rho}) \cos t-i (\rho-\bar{\rho})\sin t=2,
\end{equation}
where $\rho$ is given in \eqref{t3}. From above equation it is clear that circles $|z|=1$ and  $|z-\rho|=R $ intersects orthogonally to each other.\\
Let $\zeta=\cos t+i\sin t$ and $\bar{\zeta}=\cos t-i\sin t$ then  \eqref{t4} yields,
\begin{equation*}
\left(\frac{\zeta+\bar{\zeta}}{2}\right)(\rho+\bar{\rho}) -i\left(\frac{\zeta-\bar{\zeta}}{2i}\right)(\rho-\bar{\rho})=2, 
\end{equation*}
where $\zeta$ is the point of intersection  the circle $|z|=1$ and $|z-\rho|=R $. Thus,
\begin{equation}\label{A1}
    (\zeta+\bar{\zeta})(\rho+\bar{\rho})-(\zeta-\bar{\zeta})(\rho-\bar{\rho})=4.~~\mbox{i.e.},~~ \zeta\bar{\rho}+\bar{\zeta}\rho=2.
\end{equation}
Since $\zeta\in \partial\mathbb{D}$ therefore $\bar{\zeta}=\frac{1}{\zeta}$, putting in \eqref{A1} and we obtain
\begin{equation}\label{.1}
    \zeta=\frac{1\pm\sqrt{1-|\rho|^2}}{\bar{\rho}}.
\end{equation}
Using \eqref{t5} in \eqref{.1} we get $\displaystyle \zeta=\frac{1\pm iR}{\bar{\rho}}$.\\
More precisely, the maximum value attains at $a^+$ and is given by
\begin{eqnarray*}
    a^+&=&\frac{1+iR}{\bar{\rho}}=\frac{(z_1\bar{z_2}-z_2\bar{z_1})}{(\bar{z_2}-\bar{z_1})-\bar{z_2}\bar{z_1}(z_2-z_1)}+\frac{|z_2-z_1||1-z_1\bar{z_2}|}{(\bar{z_2}-\bar{z_1})-\bar{z_2}\bar{z_1}(z_2-z_1)}.
\end{eqnarray*}
This shows that $a^+$ is the intersecting point of $|z|=1$ and hyperbolic ray starting from $z_1$ and passing through $z_2$ which is part of the circle $|z-\rho|=R$.

\end{proof}

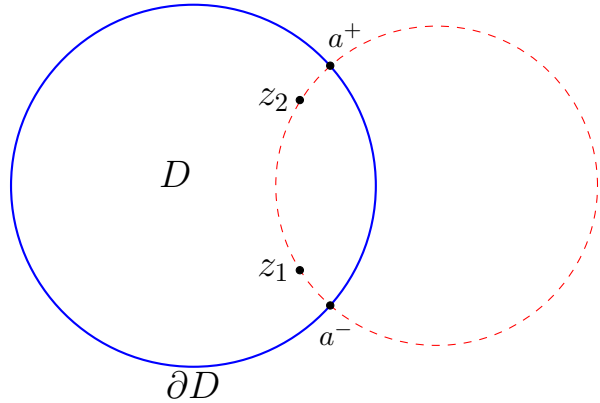
\begin{center}
\begin{tikzpicture}[scale=0.8]

\def\R{3}
\draw[thick, blue] (0,0) circle (\R);
\node at (-0.3,0.2) {\Large $D$};
\node at (0,-3.3) {\Large $\partial D$};

\def\xc{4.0}
\def\yc{0}
\pgfmathsetmacro{\rOrth}{2.6458}
\node at (2.5,2.5) {$a^+$};
\node at (2.35,-2.45) {$a^-$};
\draw[red, dashed] (\xc,\yc) circle (\rOrth);

\filldraw[black] (1.75, 1.42) circle (0.06) node[anchor=east] {\Large\textbf{$z_2$}};
\filldraw[black] (1.75, -1.4) circle (0.06) node[anchor=east] {\Large\textbf{$z_1$}};

\filldraw[black] (2.25, 1.99) circle (0.06); 
\filldraw[black] (2.25, -1.984) circle (0.06); 
\end{tikzpicture}
\end{center}
\begin{rem}
   \textnormal{ If we compute the Apollonian distance $\delta_A(z_2,z_1)$ from $z_2$ to $z_1$ by a similar way the supremum in    \eqref{t1} is attained at $a^{-}$, where $a^{-}$ is the point of intersection of $\partial\mathbb{D}$ with the hyperbolic geodesic ray starting from the point $z_2$ and passing through $z_1$.}
    \newline
   \textnormal{ More precisely $a^{-}$ is given by}
  \begin{eqnarray}
    a^{-}&=& \frac{1-iR}{\bar{\rho}}
   = \frac{(z_1\bar{z_2}-z_2\bar{z_1})}{(\bar{z_2}-\bar{z_1})-\bar{z_2}\bar{z_1}(z_2-z_1)}-\frac{|z_2-z_1||1-z_1\bar{z_2}|}{(\bar{z_2}-\bar{z_1})-\bar{z_2}\bar{z_1}(z_2-z_1)}.
\end{eqnarray}
\end{rem}
\begin{rem}
\textnormal{The above results have already been investigated by Papadopoulos and Troyanov \cite[Proposition $5.4$]{APMT} for the unit disc $\mathbb{D}\subset \mathbb{C}$ through a different technique. However by a small computation it can be checked that both the results are same.}
\end{rem}
\begin{proposition}
 \textnormal{ The Apollonian weak metric $\delta_A$ in the unit disc $\mathbb{D}$ is given by (see \cite[Theorem $2$]{APMT} )}
    \begin{equation}\label{t6}
   \delta_A(z_1,z_2) =\log M(z_1,z_2)=\log \left(\frac{|z_1-z_2|+|z_1\bar{z_2}-1|}{\big|1-|z_2|^2\big|}\right),~~\forall z_1,z_2 \in \mathbb{D}.
    \end{equation}
\end{proposition}

\begin{thm}
 \textnormal{ The Apollonian weak-Finsler structure $\mathcal{F}_A(z,\xi)$  of the Apollonian weak metric $\delta_A$ in the unit disc $\mathbb{D}$, is given by}
    \begin{equation}\label{00}
            \mathcal{F}_A(z,\xi) = \frac{|\xi|}{1-|z|^2}+\frac{\operatorname{Re}(z\bar{\xi)}}{1-|z|^2},
    \end{equation}
    \textnormal{where $z\in \mathbb{D}$  and }$\xi\in T_z\mathbb{D}$.
\end{thm}
\begin{proof}
    Taking  $z_1=z$ and $z_2=z+t\xi $ $(0 < t \in \mathbb{R})$ in $\mathbb{D}$. Then,
    \begin{equation}\label{t7}
    |z_2-z_1|=t|\xi|,
    \end{equation}
    \begin{eqnarray}
 \nonumber |z_1\bar{z_2}-1|&=&|z(\bar{z}+t\bar{\xi})-1| ={\mid|z|^2+tz\bar{\xi}-1\mid}\\  
    \nonumber &=& {(1-|z|^2)\left| 1-\frac{tz\bar{\xi}}{1-|z|^2}\right|}\\
   \label{t9} &=& (1-|z|^2)\left[1-t \operatorname{Re}\left(\frac{z\bar{\xi}}{1-|z|^2}\right)+o(t)\right].
   \end{eqnarray}
Employing  \eqref{t7} and \eqref{t9} in \eqref{t6} we obtain,\\
\begin{equation*}
     \delta_A(z,z+t\xi)=   \log\left(\frac{t|\xi|+(1-|z|^2)\left(1-t \operatorname{Re} \left(\frac{z\bar{\xi}}{1-|z|^2}\right)+o(t)\right)}{1-|z+t\xi|^2}\right).
\end{equation*}
which gives
\begin{equation*}
 \delta_A(z,z+t\xi)=   \log\left({t|\xi|+(1-|z|^2)\left(1-t \operatorname{Re} \left(\frac{z\bar{\xi}}{1-|z|^2}\right)+o(t)\right)}\right)-\log({1-|z+t\xi|^2}).
\end{equation*}
Thus, 
\begin{equation}\label{x30}
\begin{split}
\delta_A(z,z+t\xi)=& \log\Bigg[(1-|z|^2)\left({1+t  \left(\frac{|\xi|-\operatorname{Re}(z\bar{\xi})}{1-|z|^2}\right)+\frac{o(t)}{1-|z|^2}}\right)\Bigg]\\&-\log \Bigg[(1-|z|^2)\left(1-t\left(\frac{2\operatorname{Re(z\bar{\xi})}}{1-|z|^2}\right)+\frac{o(t)}{1-|z|^2}\right)\Bigg].
 \end{split}
 \end{equation}
Expanding R.H.S of \eqref{x30} and neglecting the higher order terms, we get
\begin{equation}
   \delta_A(z,z+t\xi)= \frac{t|\xi|-t\operatorname{Re}(z\bar{\xi})}{1-|z|^2}+\frac{2t\operatorname{Re}(z\bar{\xi})}{1-|z|^2}= \frac{t|\xi|+t\operatorname{Re}(z\bar{\xi})}{1-|z|^2}.
\end{equation}

Hence, {by Busemann-Mayer theorem \cite[\S 6.3]{DSSZ}, we yields}
\begin{eqnarray}\label{Fin1}
    \mathcal{F}_A(z,\xi) = \lim  \limits_{t \rightarrow 0 } \frac{\delta_A \left(z,z+t\xi\right)}{t}=\frac{|\xi|+\operatorname{Re}(z\bar{\xi})}{1-|z|^2}.
\end{eqnarray}
\end{proof}

\begin{proof}[Proof of Theorem \ref{thm01}]
Rewriting the expression of the Finsler structure obtained in (\ref{Fin1}) in real coordinates, i.e.,  $z=x=(x^1,x^2)$ and $\xi=(\xi^1,\xi^2)$, we have
\begin{equation}\label{x6}
   \mathcal{F}_A(x,\xi) = \frac{|\xi|}{1-|x|^2}+\frac{\langle x,\xi\rangle}{1-|x|^2},
\end{equation}
where $\langle . \rangle$ denotes the usual inner product in $\mathbb{R}^2$.

\noindent In view of \eqref{x6} we have $F=\alpha+\beta$. If we write  $\alpha (\xi)=\sqrt{a_{ij}\xi^i\xi^j}$, then 
  
  \begin{equation}\label{eqn2.5.118}
   a_{ij}=\frac{\delta_{ij}}{(1-|x|^2)^2},
  \end{equation}
   $\det(a_{ij})=\frac{1}{(1-|x|^2)^4}$, and the coefficients of its inverse matrix $(a_{ij})^{-1}$ are given by, 
  \begin{equation}\label{eqn2.5.119}
a^{ij}= (1-|x|^2)^2 \delta^{ij}.
  \end{equation}
 Furthermore, let $\beta(x,\xi)  =b_i(x)\xi^i$; then the coefficients $b_i(x)$ of the $1$-form $\beta$ are given by
 \begin{equation}\label{eqn2.5.120}
b_i(x)=\frac{\delta_{ij} x^j }{(1-|x|^2)}, 
 \end{equation}
 and hence \begin{equation}\label{eqn2.5.121}
||\beta||^2_{\alpha}=a^{ij}b_ib_j=|x|^2< 1.
\end{equation}
\noindent It is easy to observe that $\beta=df(x)$, where $f(x)=-\frac{1}{2}\log\left(1-|x|^2 \right).$
Thus, $\mathcal{F}_A(x,\xi)=\alpha(x,\xi)+\beta(x,\xi)$  is a Randers structure with a closed $1$-form $\beta$.
\end{proof}
\begin{rem} \label{remA}
\textnormal{ The  Apollonian weak-Finsler structure $\mathcal{F}_A$ in the unit disc $\mathbb{D}$  has closed $1$-form (see \eqref{eqn2.5.121}), and therefore,  with the help of Theorem \ref{thm1} we conclude that the geodesic trajectories of the Apollonian weak-Finsler structure  $\mathcal{F}_A$ in the unit disc $\mathbb{D}$ are the same as those of the Poincaré metric in the unit disc $\mathbb{D}$. }
\end{rem}
\begin{rem}
    \textnormal{It is well known that the Finsler structure of the Funk metric in Euclidean unit disc is given by $\mathcal{F}_F(x,\xi)=\alpha(x,\xi)+\beta(x,\xi),$ where $\alpha(x,\xi)=\frac{\sqrt{|\xi|^2-(|x|^2|\xi|^2-\langle x,\xi \rangle^2)}}{1-|x|^2}$
    is the Klein metric on the unit disc and $\beta(x,\xi)=\frac{\langle x,\xi\rangle}{(1-|x|^2)}$. That is, the Randers structure $\mathcal{F}(x,\xi)=\alpha(x,\xi)+\beta(x,\xi)$ corresponds to the Finsler structure of the Apollonian weak metric if $\alpha$ is the Poincar\'e metric on unit disc and that of the Funk metric if $\alpha$ is the Klein metric on the unit disc with the same one form $\beta=\frac{\langle x,\xi\rangle}{(1-|x|^2)}$ }. 
\end{rem}
\begin{rem}
   \textnormal{The arithmetic symmetrization of Apollonian weak metric on a convex set is a semi metric called Barbilian metric. It is denoted by $S\delta_{A}$}
    \begin{equation}
       S\delta_A(z_1,z_2)=\frac{1}{2} \left( \delta_A(z_1,z_2)+\delta_A(z_2,z_1) \right)=\frac{1}{2}\log \left(\frac{|z_1\bar{z_2}-1|+|z_2-z_1|}{|z_1\bar{z_2}-1|-|z_2-z_1|}\right).
    \end{equation}
     \textnormal{ However the Finsler structure of the Barbilian metric on the unit disc $\mathbb{D}$ is a Riemannian metric which coincides with the Poincar\'e metric of the unit disc and given by} 
    \begin{equation}
        S\mathcal{F}_A(z,\xi)=\frac{1}{2} \left( \mathcal{F}_A(z,\xi)+\mathcal{F}_A(z,-\xi) \right)=\frac{|\xi|}{1-|z|^2}.
        \end{equation}
\end{rem}
\begin{que}
   \textnormal{ It is an open question to find all convex set on which the Finsler structure of the Barbilian metric is Riemannian?}
\end{que}

\begin{proof}[Proof of Theorem \ref{thm3.}] The indicatrix for the Apollonian weak-Finsler structure $\mathcal{F}_A$ is given by 
\begin{equation*}
S_x = \left\{\xi \in T_x\mathbb{D}:\mathcal{F}_A(\xi)=1 \right\}.
\end{equation*}
Thus by \eqref{x6} we have 
\begin{eqnarray}\label{m1}
  \frac{|\xi|}{1-|x|^2}+\frac{\langle x,\xi\rangle}{1-|x|^2}=1, ~ \forall~\xi\in T_x\mathbb{D}.
\end{eqnarray}
Since $\xi \in T_x\mathbb{D} $, let us write $\eta=x+\xi$. Therefore
the above equation can be rewritten as,
 \begin{eqnarray} \label{x4}
    |\eta -x|=(1-|x|^2)-\langle x,\eta -x\rangle. 
\end{eqnarray}
Squaring on both sides and simplifying,  we obtain
\begin{equation}\label{x50}
    \eta_1^2(1-x_1^2)+\eta_2^2(1-x_2^2)-2x_1x_2\eta_1\eta_2=1-x_1^2-x_2^2.
\end{equation}
Comparing \eqref{x50} with the general equation of the conic $A(\eta^1)^2+B\eta^1\eta^2+C(\eta^2)^2+D\eta^1+E\eta^2+F=0$, shows that the indicatrix of given Finsler structure $\mathcal{F}_A$  at the point $x \in \mathbb{D}$ is an ellipse, as
\begin{eqnarray}
\nonumber   B^2-4AC= 4(x^1)^2(x^2)^2-4(1-(x^1)^2)(1-(x^2)^2)=-4(1-|x|^2)< 0.
\end{eqnarray}
Further, from \eqref{x50}, it can be deduced that  eccentricity  of the ellipse (indicatrix) is $|x|$, foci of the ellipse are $x$ and $-x$, centered at the origin with its major axis as the line joining origin to the point $x$.
\end{proof}

\begin{tikzpicture}[scale=3]
    \draw[->] (-1.4, 0) -- (1.4, 0) node[right] {$x$};
    \draw[->] (0, -1.2) -- (0, 1.2) node[above] {$y$};

    \draw[blue, thick] (0, 0) circle(1);
    \node[blue] at (-0.85, 0.85) {$D$};

    \newcommand{\drawImplicitEllipse}[4]{%
        \def\xone{#1}
        \def\xtwo{#2}
        \def\labeltext{#3}
        \def\ellipColor{#4}

        \pgfmathsetmacro{\qxx}{1 - \xone*\xone}
        \pgfmathsetmacro{\qyy}{1 - \xtwo*\xtwo}
        \pgfmathsetmacro{\qxy}{- \xone * \xtwo}
        \pgfmathsetmacro{\rhs}{1 - \xone*\xone - \xtwo*\xtwo}

        \pgfmathsetmacro{\trace}{\qxx + \qyy}
        \pgfmathsetmacro{\detq}{\qxx*\qyy - \qxy*\qxy}
        \pgfmathsetmacro{\temp}{sqrt((\qxx - \qyy)^2 + 4*\qxy*\qxy)}

        \pgfmathsetmacro{\lambdaone}{(\trace + \temp)/2}
        \pgfmathsetmacro{\lambdatwo}{(\trace - \temp)/2}

        \pgfmathsetmacro{\a}{sqrt(\rhs / \lambdaone)}
        \pgfmathsetmacro{\b}{sqrt(\rhs / \lambdatwo)}

        \pgfmathsetmacro{\theta}{0}
        \ifdim \qxy pt=0pt
            \pgfmathsetmacro{\theta}{0}
        \else
            \pgfmathsetmacro{\theta}{0.5 * atan2(2 * \qxy, \qxx - \qyy)}
        \fi

        \filldraw[red] (\xone,\xtwo) circle(0.01);
        \node[above right] at (\xone,\xtwo) {\labeltext};

        \draw[thick, \ellipColor, samples=200, smooth, variable=\t, domain=0:360]
            plot
            (
                { \a * cos(\t)*cos(\theta) - \b * sin(\t)*sin(\theta) },
                { \a * cos(\t)*sin(\theta) + \b * sin(\t)*cos(\theta) }
            );
    }

    \drawImplicitEllipse{0.3}{0.3}{A}{red}
    \drawImplicitEllipse{0.5}{0.5}{B}{orange}
    \drawImplicitEllipse{0.68}{0.68}{C}{violet}

    \node[black!60!black, align=center] at (0, -1.4)
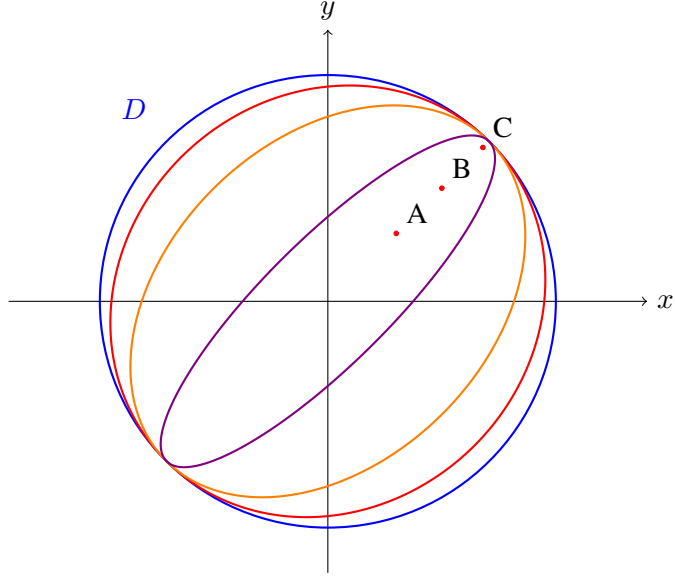
   {Indicatrices of Apollonian weak-Finsler structure\\
    of unit disc $\mathbb{D}$ at points A = (0.3,0.3), B =(0.5,0.5), and C = (0.68,0.68)};
   \end{tikzpicture}

\begin{rem}
\textnormal{Thus the Finsler structure of Apollonian weak metric on the unit disc $\mathbb{D}$ is a family of ellipses as its indicatrices, one in each tangent space with one of its foci at the origin of the tangent space(as shown in above figure). 
It is interesting to note that as we move the point $x$ towards the boundary of the disc $\mathbb{D}$, the indicatrix ellipse $S_x$ becomes thiner and thiner.}
\end{rem}
\section{Some curvatures of the Apollonian weak-Finsler structure }\label{sec4}
In this section, we explicitly obtain the expressions for the $S$-curvature, the Riemann curvature, the Ricci curvature and the flag curvature of the Apollonian weak-Finsler structure $\mathcal{F}_A$ 
on the unit disc $\mathbb{D}$.
\subsection{Spray coefficients and $S$-curvatures of the Apollonian weak-Finsler structure $\mathcal{F}_A$}\label{sb3}
In this subsection,
we recall the formula for $S$-curvature of 
a general Randers structure $F=\alpha+\beta$, where $\alpha(x, \xi)=\sqrt{a_{ij}\xi^i\xi^j}$ and $\beta(x, \xi)=b_i(x)\xi^i$.\\
Let  $\bar{\Gamma}^k_{ij}(x)$  denote the  Christoffel symbols of Riemannian metric $\alpha$. Then we have,
 \begin{equation}\label{eqnn 4.3.48}
 b_{i|j}:=\frac{\partial b_i}{\partial x^j}-b_k \bar{\Gamma}^k_{ij}.
 \end{equation}
 We introduce the following notations,
 \begin{equation}\label{eqnn 4.3.50}
  r_{ij}:=\frac{1}{2}\left( b_{i|j}+b_{j|i}\right),~~ s_{ij}:=\frac{1}{2}\left( b_{i|j}-b_{j|i}\right).
  \end{equation}
 \begin{equation}\label{eqnn 4.3.51}
   s^i_j:=a^{ih}s_{hj},~~ s_j:=b_i s^i_j=b^j s_{ij}, ~~r_j:=b^i r_{ij}, ~~b^j=a^{ij} b_i,
   \end{equation}
\begin{equation}\label{eqnn 4.3.52}
e_{ij}:=r_{ij}+b_is_j+b_js_i,
\end{equation}
\begin{equation*}e_{00}:=e_{ij}\xi^i\xi^j,~~ s_0:=s_i\xi^i ~~ \text{and}~~ s^i_0:=s_j^i\xi^j.\end{equation*}\\
Now consider, 
 \begin{equation}\label{eqnA0}
      \rho:=\log \sqrt{\left( 1-||\beta||^2_\alpha\right)},~\mbox{and}~ \rho_0:=\rho_i\xi^i,\rho_i:=\rho_{x^i}(x).
  \end{equation}
It is well known that the $S$-curvature of the Randers structure $F=\alpha +\beta $ is given by,
\begin{equation}\label{eqnn 4.65}
 S=(n+1)\Big[\frac{e_{00}}{2F}-(s_0+\rho_0) \Big],
 \end{equation}
see \cite[$\S 3.2$]{CXSZ} for more details.

\begin{thm}\label{thm4.1}
    \textnormal{The  Apollonian weak-Finsler structure $\mathcal{F}_A$ on the disc $\mathbb{D}$,  given by \eqref{x6}, is projectively flat and its $S$-curvature is given by:}
   \begin{equation*}
      \textbf{S} (x,\xi)= \frac{3|\xi|\Big[ (1+|x|^2)|\xi|+2\langle x,\xi\rangle\Big]}{2\mathcal{F}_A(1-|x|^2)^2},~\forall (x,\xi) \in T\mathbb{D}
    \end{equation*}
\end{thm}
\begin{proof}
From equation \eqref{eqnn 4.65}, to calculate the $S$-curvature of the Apollonian weak-Finsler structure $\mathcal{F}_A$,  we proceed as follows. The Christoffel symbols $\bar{\Gamma}^k_{ij}(x)$ of Riemannian  metric $\alpha$  are given by
 \begin{eqnarray}
\label{eqn0}  \bar{\Gamma}^k_{ij}(x)=\frac{2\left(\delta_{ki}x^j+\delta_{kj}x^i-\delta_{ij}x^k \right)}{(1-|x|^2)},~~ \mbox{for all}~~ x\in \mathbb{D}.
  \end{eqnarray}
   Clearly,
  \begin{equation}\label{eqn01} 
 \bar{\Gamma}^k_{ij}(x)=\bar{\Gamma}^k_{ji}(x).
  \end{equation}
Using expression for $b_i$ from \eqref{eqn2.5.120}, we get,
   \begin{equation}\label{eqnA1}
    \frac{\partial b_i}{\partial x^j}=\frac{\partial b_j}{\partial x^i}=\frac{(1-|x|^2)\delta_{ij}+2x^ix^j}{(1-|x|^2)^2}.
\end{equation}
Substituting \eqref{eqn01} and \eqref{eqnA1}
in \eqref{eqnn 4.3.48}, we obtain
\begin{equation}\label{eqnA2}
   b_{i|j} =b_{j|i}=\frac{(1+|x|^2)\delta_{ij}-2x^ix^j}{(1-|x|^2)^2}.
\end{equation}
Employing \eqref{eqnA2} in \eqref{eqnn 4.3.50}  yields,
\begin{equation}\label{eqnA3}
  s_{ij}=0,~ r_{ij}=b_{i|j}=\frac{(1+|x|^2)\delta_{ij}-2x^ix^j}{(1-|x|^2)^2},~\forall i,j=1,2.
\end{equation}
Applying  \eqref{eqnA3} in \eqref{eqnn 4.3.51} and \eqref{eqnn 4.3.52}, we have
 \begin{equation}\label{eqnA4}
  s^i_j=0,~ s_j=0, ~ \mbox{and}~~ e_{ij}=r_{ij}=\frac{(1+|x|^2)\delta_{ij}-2x^ix^j}{(1-|x|^2)^2}.
 \end{equation}
Let $G^i=G^i(x,\xi)$ and $\bar{G}^i=\bar{G}^i(x,\xi)$ denote the spray coefficients of $\mathcal{F_A}$ and $\alpha$ respectively. Then  $G^i$ and $\bar{G}^i$ are related by (see \cite[\S 2.3, equation (2.19)]{CXSZ})
\begin{equation}\label{eqnA5}
G^i=\bar{G}^i+P \xi^i+Q^i,
\end{equation}  
where
\begin{equation}\label{eqnA6}
P:=\frac{e_{00}}{2\mathcal{F}_A}-s_0, ~~ ~~  Q^i:=\alpha s^i_0  ~~ \mbox{and}~~ \bar{G}^i=\frac{1}{2}\bar{\Gamma}^i_{jk}\xi^j\xi^k,
\end{equation}
and  where $e_{00}:=e_{ij}\xi^i\xi^j$, $s_0:=s_i\xi^i$ and $s^i_0:=s_j^i \xi^j$.\\
 Therefore from \eqref{eqn0} , \eqref{eqnA4} and \eqref{eqnA6} we find  that,
 \begin{equation}\label{eqnA7}
 P=\frac{e_{00}}{2\mathcal{F}_A}=\frac{r_{ij}\xi^i\xi^j}{2\mathcal{F}_A}=\frac{(1+|x|^2)|\xi|^2-2\langle x,\xi\rangle^2}{2\mathcal{F}_A(1-|x|^2)^2},
 \end{equation}
 \begin{equation}\label{eqnA8}
 Q^i=0 ~ \mbox{and}~~ \bar{G}^i=\frac{\Big[ 2\xi^i\langle x, \xi \rangle-|\xi|^2 x^i\Big]}{1-|x|^2},~  i=1,2.
 \end{equation}
 Employing \eqref{eqnA7}, \eqref{eqnA8} in  \eqref{eqnA5}, we see that the spray coefficients are given by
 \begin{equation}\label{eqnA9}
 G^i=\frac{\Big[ 2\xi^i\langle x, \xi \rangle-|\xi|^2 x^i\Big]}{1-|x|^2}+\frac{(1+|x|^2)|\xi|^2-2\langle x,\xi\rangle^2}{2\mathcal{F}_A(1-|x|^2)^2}\xi^i. 
 \end{equation}
 Therefore, the  Apollonian weak-Finsler structure $\mathcal{F}_A$ is projectively flat (see Proposition \ref{ppn2.3}).\\
 From \eqref{eqn2.5.121} and \eqref{eqnA0} for $\mathcal{F}_A$ we have 
 \begin{equation}\label{rho}
     \rho=\frac{1}{2}\log\left(1-|x|^2 \right).
 \end{equation}
 Therefore,
 \begin{equation}\label{rho0}
     \rho_0=\rho_{x^i}\xi^i=-\frac{\langle x,\xi \rangle}{\left(1-|x|^2 \right)}.
 \end{equation}
Availing \eqref{rho0} and \eqref{eqnA7} in \eqref{eqnn 4.65} for $n=2$, we obtain $S$-curvature as:
 \begin{equation*}
      \textbf{S}= \frac{3|\xi|\Big[ (1+|x|^2)|\xi|+2\langle x,\xi\rangle\Big]}{2\mathcal{F}_A(1-|x|^2)^2}.
    \end{equation*}
\end{proof}

\begin{proof}[Proof of Theorem \ref{thm1.3}] From Theorem \ref{thm4.1} we have
    \begin{eqnarray*}
         S-\frac{3}{2}\mathcal{F}_A&=&\frac{3|\xi|\Big[ (1+|x|^2)|\xi|+2\langle x,\xi\rangle\Big]}{2\mathcal{F}_A(1-|x|^2)^2}-\frac{3}{2}\Big[\frac{|\xi|}{1-|x|^2}+\frac{\langle x,\xi\rangle}{1-|x|^2}\Big]\\
         &=&\frac{3\Big[|x|^2|\xi|^2-\langle x,\xi\rangle^2\Big]}{2\mathcal{F}_A(1-|x|^2)^2}\geq 0. 
    \end{eqnarray*}
     Hence, $$S \geq \frac{3}{2}\mathcal{F}_A.$$
\end{proof}
\subsection{Ricci and Flag curvature of the Apollonian weak-Finsler structure $\mathcal{F}_A$}\label{sb4}
In this subsection we explicitly find out the Ricci and Flag curvature of the Apollonian weak-Finsler structure $\mathcal{F}_A$.
\begin{thm}\label{thm4.2}
\textnormal{Let  $\mathcal{F}_A$ be the Apollonian weak-Finsler structure on the unit disc $\mathbb{D}$ given by \eqref{x6}, then the Riemann curvature $R^i_k$, the Ricci curvature  $\textbf{Ric}$, and the flag curvature $\textbf{K}$ of  $\mathcal{F}_A$  are respectively given by}
\begin{equation}
 R^i_k=- 4\left(  \delta^i_k \alpha^2-\alpha \alpha_k\xi^i \right)+\Bigg[ 3\left( \frac{\phi}{2\mathcal{F}_A}\right)^2  -\frac{\psi}{2\mathcal{F}_A}\Bigg]\left( \delta^i_k-\frac{(\mathcal{F}_A)_{\xi^k}}{\mathcal{F}_A}\xi^i\right) +\tau_k \xi^i,
 \end{equation}
  \textnormal{with}
$$\phi=\frac{(1+|x|^2)|\xi|^2-2\langle x,\xi\rangle^2}{(1-|x|^2)^2},$$
$$\psi=\frac{-2\left(1+3|x|^2 \right)|\xi|^2\langle x,\xi \rangle+8\langle x,\xi \rangle^3}{\left(1-|x|^2 \right)^3}, ~~\textnormal{and} ~~~ \tau_k=\frac{4\Big[|\xi|^2 x^k-\langle x,\xi \rangle \xi^k\Big]}{\mathcal{F}_A\left(1-|x|^2 \right)^3},$$
\begin{equation*}
    \textbf{Ric}=\frac{\Big[3(1+|x|^2)^2-16\Big]|\xi|^4+(12|x|^2-28)|\xi|^3\langle x,\xi\rangle-24|\xi|^2\langle x,\xi\rangle^2-16|\xi|\langle x,\xi\rangle^3-4\langle x,\xi\rangle^4}{4\mathcal{F}_A^{2}(1-|x|^2)^4},
\end{equation*}
\textnormal{and}
\begin{equation*}\label{.00}
  \textbf{K}  =\frac{\Big[3(1+|x|^2)^2-16\Big]|\xi|^4+(12|x|^2-28)|\xi|^3\langle x,\xi\rangle-24|\xi|^2\langle x,\xi\rangle^2-16|\xi|\langle x,\xi\rangle^3-4\langle x,\xi\rangle^4}{4\mathcal{F}_A^{4}(1-|x|^2)^4}.
\end{equation*}

 \end{thm}
 \begin{proof} The Riemannian curvature of the Randers structure $F=\alpha+\beta$ with closed $1$-form $\beta$ on an $n$-dimensional manifold is given by (see \cite[$\S 5.2$, equation $(5.10)$]{CXSZ})
 \begin{equation}\label{eqnn 4.3.70}
 R^i_k=\overline{R^i_k}+\Bigg[ 3\left( \frac{\phi}{2F}\right)^2  -\frac{\psi}{2F}\Bigg]\left( \delta^i_k-\frac{F_{\xi^k}}{F}\xi^i\right) +\tau_k \xi^i,
 \end{equation}
 
where
\begin{equation}\label{eqnn 4.3.72}
\phi:=b_{i|j}\xi^i\xi^j,~~~\psi:=b_{i|j|k}\xi^i\xi^j\xi^k,~~~\tau_k:=\frac{1}{F}\left( b_{i|j|k}-b_{i|k|j}\right) \xi^i\xi^j,
\end{equation}
and 
\begin{equation}\label{eqnn 4.3.720}
 b_{i|j|k}=\frac{\partial b_{i|j}}{\partial x^k}-b_{i|m}\bar{\Gamma}^m_{jk}-b_{j|m}\bar{\Gamma}^m_{ik}.   
\end{equation}
Here $\overline{R^i_k}$   denotes the Riemann curvature of the Riemannian metric $\alpha$.\\
It is well known that the Gaussian curvature of the Riemannian metric $\alpha=\frac{|\xi|}{1-|x|^2}$   is $-4$. Therefore, from \eqref{eqn2.1.12} the Riemann curvature of the  metric $\alpha$ is given by 
 \begin{equation}\label{eqnn 4.3.710}
\overline{R^i_k}=- 4\left(  \delta^i_k \alpha^2-\alpha \alpha_k\xi^i \right),~ \alpha_k:=\frac{\partial \alpha}{\partial \xi ^k}.
\end{equation} 
Considering \eqref{eqnn 4.3.70} with \eqref{eqnn 4.3.710}, we get the desired expression for the Riemann curvature.\\
In view of \eqref{eqnA3}, \eqref{eqnn 4.3.720} for the Apollonian weak-Finsler structure $\mathcal{F}_A$ in dimension $2$, the functions $\phi$ and $\psi$ can be explicitly calculated as follows:
 \begin{eqnarray}
\label{eqnn 4.3.74} \phi&=&b_{i|j}\xi^i\xi^j=(1+|x|^2)\alpha^2-2\beta^2=\frac{(1+|x|^2)|\xi|^2-2\langle x,\xi\rangle^2}{(1-|x|^2)^2},
\end{eqnarray}
\begin{eqnarray}
\mbox{and}~~~~\nonumber \psi&=&b_{i|j|k}\xi^i\xi^j\xi^k=\left(\frac{\partial b_{i|j}}{\partial x^k}-b_{i|m}\bar{\Gamma}^m_{jk}-b_{j|m}\bar{\Gamma}^m_{ik}\right)\xi^i\xi^j\xi^k\\
\label{eqnn 4.3.75} &=&-2\left(1+3|x|^2 \right)\alpha^2\beta+8\beta^3=\frac{-2\left(1+3|x|^2 \right)|\xi|^2\langle x,\xi \rangle+8\langle x,\xi \rangle^3}{\left(1-|x|^2 \right)^3}.
\end{eqnarray}
Here, $\displaystyle \alpha=\frac{|\xi|}{(1-|x|^2)}$ and $\displaystyle \beta=\frac{\langle x,\xi \rangle}{(1-|x|^2)}$ (see Theorem \ref{thm01}).\\

\noindent Applying \eqref{eqnA3}, \eqref{eqnn 4.3.720} in \eqref{eqnn 4.3.72}, we obtain
\begin{eqnarray}
 \label{eqnn 4.3.76} \tau_k=\frac{4\Big[\langle x,\xi \rangle \xi^k-|\xi|^2 x^k\Big]}{\mathcal{F}_A\left(1-|x|^2 \right)^3}.
\end{eqnarray}
Further,  the Ricci curvature of the Randers structure $F=\alpha+\beta$ with $\beta$ closed $1$-form is given by (see \cite[$\S 5.2$, equation $(5.12)$]{CXSZ})
\begin{equation}\label{eqnn 4.3.71}
\textbf{Ric} =\overline{\textbf{Ric}} +(n-1)\Bigg[3\left( \frac{\phi}{2F}\right)^2  -\frac{\psi}{2F} \Bigg],
\end{equation} 
where $\overline{\textbf{Ric}}$   denotes the Ricci curvature of the Riemannian metric $\alpha$.

\noindent Here, since $\alpha$ is the Poincar\'e metric on the unit disc, the Ricci curvature of the Riemannian metric $\alpha$ is given by
\begin{equation}\label{eqnn 4.3.711}
\overline{\textbf{Ric}}=\overline{R^i_i}=-4\alpha^2.
\end{equation} 
And consequently, the Ricci curvature of the
the Apollonian weak-Finsler structure $\mathcal{F}_A$ in dimension $2$ is given by
\begin{equation}\label{eqnn 4.3.712}
\textbf{Ric}= R^i_i=-4\alpha^2+\Bigg[ 3\left( \frac{\phi}{2\mathcal{F}_A}\right)^2  -\frac{\psi}{2\mathcal{F}_A}\Bigg],
 \end{equation}
 where $\phi$ , $\psi$ are respectively given by \eqref{eqnn 4.3.74} and \eqref{eqnn 4.3.75}. 
 After simplification, which can be rewritten as

\begin{eqnarray}
\nonumber \textbf{Ric}&=&\frac{3{\phi^2}-{2\psi}{\mathcal{F}_A-16{\alpha^2}{\mathcal{F}_A^2}}}{4\mathcal{F}_A^2}\\
\nonumber &=&\frac{3\Big[(1+|x|^2)\alpha^2-2\beta^2\Big]^2 -2\Big[-2\left(1+3|x|^2 \right)\alpha^2\beta+8\beta^3\Big]\left(\alpha+\beta \right)-16\alpha^2 \left(\alpha+\beta \right)^2}{4\mathcal{F}_A^{2}}\\
\label{Ric} &=&\frac{\Big[3(1+|x|^2)^2-16\Big]\alpha^4+(12|x|^2-28)\alpha^3\beta-24\alpha^2\beta^2-16\alpha\beta^3-4\beta^4}{4\mathcal{F}_A^{2}}.
\end{eqnarray}
which gives
\begin{equation}\label{eqnn 4.3.79}
    \textbf{Ric}=\frac{\Big[3(1+|x|^2)^2-16\Big]|\xi|^4+(12|x|^2-28)|\xi|^3\langle x,\xi\rangle-24|\xi|^2\langle x,\xi\rangle^2-16|\xi|\langle x,\xi\rangle^3-4\langle x,\xi\rangle^4}{4\mathcal{F}_A^{2}(1-|x|^2)^4}.
\end{equation}
In view of equation \eqref{eqn2.1.12} and \eqref{eqnn 4.3.79}, the flag curvature of the Apollonian weak-Finsler structure $\mathcal{F}_A$ is given by
\begin{equation*}
  \textbf{K}  =\frac{\Big[3(1+|x|^2)^2-16\Big]|\xi|^4+(12|x|^2-28)|\xi|^3\langle x,\xi\rangle-24|\xi|^2\langle x,\xi\rangle^2-16|\xi|\langle x,\xi\rangle^3-4\langle x,\xi\rangle^4}{4\mathcal{F}_A^{4}(1-|x|^2)^4}.
\end{equation*}
\end{proof}

\begin{proof}[Proof of the Theorem \ref{thm1.4}]
Since the flag curvature $K$ and Ricci curvature $\textbf{Ric}$ of the Apollonian weak-Finsler structure $\mathcal{F}_A$ are related by $\textbf{K}=\frac{\textbf{Ric}}{\mathcal{F}_A^2}$. In view of \eqref{Ric}, we obtain
\begin{eqnarray}
 \nonumber   \textbf{K}+1&=&\frac{\textbf{Ric}}{\mathcal{F}_A^2}+1\\
\nonumber    &=&\frac{\Big[3(1+|x|^2)^2-16\Big]\alpha^4+(12|x|^2-28)\alpha^3\beta-24\alpha^2\beta^2-16\alpha\beta^3-4\beta^4+4(\alpha+\beta)^4}{4\mathcal{F}_A^4}\\
\nonumber &=&\frac{3\alpha^2[(1+|x|^2)\alpha+2\beta ]^2-12\alpha^2\mathcal{F}_A^2}{4\mathcal{F}_A^4}\\
&=& \frac{3\alpha^2[\{(1+|x|^2)\alpha+2\beta \}^2-4\mathcal{F}_A^2]}{4\mathcal{F}_A^4}.
\end{eqnarray}
Since,  $|x| < 1$, therefore $\{(1+|x|^2)\alpha+2\beta\}^2-4\mathcal{F}_A^2<0$ .\\
Hence,
$$\textbf{K} +1 < 0, ~\mbox{i.e.},~ \textbf{K} < -1.$$
Similarly, applying \eqref{Ric}, $\textbf{K}$ can be rewritten as
\begin{eqnarray*}
\nonumber    \textbf{K}  &=&\frac{3\alpha^2 \Big[(1+|x|^2)\alpha+2\beta \Big]^2-4\mathcal{F}_A^4-12\alpha^2\mathcal{F}_A^2}{4\mathcal{F}_A^4}.\\
\end{eqnarray*}

Taking $\theta$ as the angle between the position vector $x$ and tangent vector $\xi$, we have $\langle x,\xi\rangle=|x||\xi|\cos{\theta}$, and using $\displaystyle \alpha=\frac{|\xi|}{(1-|x|^2)}$ and $\displaystyle \beta=\frac{\langle x,\xi \rangle}{(1-|x|^2)}=\frac{|x||\xi|\cos{\theta}}{1-|x|^2}$, the above expression can be rewritten as
\begin{equation*}
    \textbf{K}=\frac{3(1+|x|^2+2|x|\cos\theta)^2-4(1+|x|\cos\theta)^4-12(1+|x|\cos\theta)^2}{4(1+|x|\cos\theta)^4},
\end{equation*}
that is,
\begin{equation} \label{K}
   \textnormal{K}= \frac{3(1+|x|^2+2|x|\cos\theta)^2}{4(1+|x|\cos\theta)^4}-\frac{3}{(1+|x|\cos{\theta})^2}-1 .
\end{equation}
Now if we take the angle between $x$ and $\xi$ as $180^{\circ}$, i.e.,  the vector $\xi$ in the opposite direction of the position vector $x$, we have \eqref{K}
\begin{equation}
     \textbf{K} = \frac{3}{4}-\frac{3}{(1-|x|)^2}-1.
\end{equation}
Clearly as $|x|$ approaches to $1$, i.e., as the point $x$ moves towards boundary the flag curvature $\textbf{K}$ approaches to $-\infty$.

\end{proof}

\begin{rem}
  \textnormal{  In view of \eqref{K}, it is clear that the both bounds of the flag curvature will never be attained. In fact, both bounds will be approached towards the boundary of the disc. More precisely, when the position vector and the tangent vector are orthogonal, the curvature has the maximum limit $-1$ and when the position vector and the tangent vector are in opposite direction the curvature will approach very fast towards $-\infty$, which is the key feature of the flag curvature that it depends on the point as well as on the direction.}
\end{rem}
\section{Some other aspects of the Apollonian weak-Finsler structure}\label{sec5}
In this section, we discuss the geometric realization of the Apollonian weak-Finsler structure on the unit disc.
We also investigate the associated Zermelo navigation data for this structure on the unit disc $\mathbb{D}$.
\subsection{The realization of Apollonian weak-Finsler structure on the upper sheet  of the hyperboloid of two sheets}
In this subsection, we show that the Apollonian weak-Finsler structure on the unit disc can be realized as the pullback of a non-positive definite Randers structure in the upper half space on the upper sheet of the hyperboloid of two sheets. \\\\
Let $\mathbb{R}_+^3 = \left\lbrace(\tilde{x}^1,\tilde{x}^2, \tilde{x}^3)\in \mathbb{R}^3 : \tilde{x}^3 > 0\right\rbrace $ be the upper half space with the Lorentzian  metric $\alpha_L$, defined by $\alpha_L(\tilde{x},\tilde{\xi})=\sqrt{(\tilde{\xi}^1)^2+(\tilde{\xi}^2)^2-(\tilde{\xi}^3)^2}$ with $\tilde{x}\in \mathbb{R}_+^3 $ and $\tilde{\xi} \in T_{\tilde{x}}\mathbb{R}_+^3\cong \mathbb{R}^3$ and a $1$-form $\beta_L=\frac{1}{1+\tilde{x}^3}d\tilde{x}^3$ in $\mathbb{R}_+^3$. Now consider the deformation $F_L$ of   $\alpha_L$ by  the $1$-form $\beta_L=\frac{1}{1+\tilde{x}^3}d\tilde{x}^3$ in $\mathbb{R}_+^3$ as follows: $F_L(\tilde{x},\tilde{\xi}) =\alpha_L(\tilde{x},\tilde{\xi})+\beta_L(\tilde{x},\tilde{\xi})$. We parametrize the upper half portion $\mathbb{H_+} = \left\lbrace(\tilde{x}^1, \tilde{x}^2, \tilde{x}^3)\in \mathbb{R}^3 : \tilde{x}^3 = \sqrt{1+(\tilde{x}^1)^2+(\tilde{x}^2)^2} \right\rbrace,$ 
   of the hyperboloid of two sheets in $\mathbb{R}^3$ as follows:
\begin{equation}\label{eqn 3.A3}
\pi : \mathbb{D} \rightarrow \mathbb{H}_+,~~~~~~\pi(x)=\left( \frac{2x}{{1-|x|^2}},\frac{1+|x|^2}{1-|x|^2}\right) . 
\end{equation}
\begin{proposition}
   \textnormal{ The pullback of the metric $F_L$ deﬁned as above, on the upper
sheet of the hyperboloid of two sheets, by the map $\pi$ is the realization of the Apollonian weak-Finsler structure on the unit disc , that is, $\pi^*F_L=\mathcal{F}_A$}.
\end{proposition}
\begin{proof}
  In view of \eqref{eqn 3.A3}\\
  \begin{equation*}
      \pi^1(x) = \frac{2x^1}{{1-|x|^2}} , ~ \pi^2(x) = \frac{2x^2}{{1-|x|^2}} ~ , ~ \pi^3(x) = \frac{1+|x|^2}{{1-|x|^2}} .
 \end{equation*}
 Therefore,
 \begin{eqnarray*}
     d\pi^1&=&\frac{2}{(1-|x|^2)^2}\left[ \left\lbrace 1- |x|^2+2(x^1)^2\right\rbrace dx^1 + 2x^1x^2 dx^2\right],\\
     d\pi^2&=&\frac{2}{(1-|x|^2)^2}\left[  2x^1x^2 dx^1 + \left\lbrace 1- |x|^2+2(x^2)^2\right\rbrace dx^2\right],\\
     d\pi^3&=&\frac{2}{(1-|x|^2)^2}\left[ 2 x^1 dx^1 +2x^2 dx^2\right]. 
 \end{eqnarray*}
Hence,
\begin{eqnarray}
   \nonumber  \pi^*F_L(x,\xi)&=&\frac{1}{2}\left( \sqrt{(d\pi^1)^2+(d\pi^2)^2-(d\pi^3)^2}+\frac{1}{1+\pi^3}d\pi^3\right)(x,\xi)\\
 \nonumber   &=&\frac{|\xi|}{1-|x|^2}+\frac{\langle x ,  \xi \rangle}{1-|x|^2} = \mathcal{F}_A.
\end{eqnarray}
\end{proof}

\noindent
Thus, we have shown  that the pullback of  $F_L$ on the upper half of the hyperboloid of two sheets $\mathbb{H}_+$ gives  the Apollonian weak-Finsler structure on the unit disc .

\subsection{Zermelo navigation description of Apollonian weak-Finsler structure}
It is well known that any Randers structure on a manifold $M$ has a Zermelo navigation representation. For instant, if $F=\alpha+\beta$ is given the Randers structure with $\alpha=\sqrt{a_{ij}(x)\xi^i\xi^j}$ and differential $1$-form $\beta=b_i(x)\xi^i$, satisfying $||\beta||^2_\alpha=a^{ij}b_ib_j < 1$. Then the Zermelo Navigation for this Randers structure is the triple $(M,h,W)$, where  $h=\sqrt{h_{ij}\xi^i \xi^j}$ with
\begin{equation*}
    h_{ij}= c (a_{ij}-b_ib_j), ~~ W^i=-\frac{b^i}{c},~ b^i=a^{ij}b_j~ \mbox{and}~ c=1-||\beta||^2_\alpha.
\end{equation*}
Moreover, $||W||_h=||\beta||_{\alpha}$.\\
Also given the Zermelo data, we can get back the Randers structure. 
And this 1-1 correspondence is useful in finding the geodesics of the Randers structure. 
 See for more details  \cite[ Example $1.4.3$]{SSZ}.\\\\
In this subsection, we obtain the Zermelo data for the Apollonian weak-Finsler structure $\mathcal{F}_A$, which is a trivially a Randers structure.
We have,
 \begin{equation*}
    \mathcal{F}_A(x, \xi) = \frac{|\xi|}{1-|x|^2}+\frac{ \langle x,\xi \rangle}{(1-|x|^2)}.
   \end{equation*}
We need to find  $h_{ij}, W^i$ defined above. From 
\eqref{eqn2.5.121},  we have,  
\begin{equation}
c=  1- ||\beta||^2_{\alpha}=1-|x|^2.  
\end{equation}
Employing \eqref{eqn2.5.118} and \eqref{eqn2.5.120}, we see  
\begin{equation}\label{eqn 4.100A}
       h_{ij} =\frac{\delta_{ij}-x^ix^j}{1-|x|^2}.
 \end{equation} 
Clearly,
 \begin{equation}\label{eqn 4.101A}
   W=\left( W^i \right)=\left(- x^i  \right).
\end{equation}
Also, 
$$||W||^2_h=||\beta||^2_{\alpha}=|x|^2.$$
 Thus, we have 
 \begin{proposition}
    \textnormal{ The Zermelo Navigation data for the Apollonian weak-Finsler structure $\mathcal{F}_A$ on the unit disc $\mathbb{D}$ is given by $\left( \mathbb{D}, h, W\right)$, where the components of the Riemannian metric $h$ is given by \eqref{eqn 4.100A} and that of the vector field by \eqref{eqn 4.101A}.}
 \end{proposition}

\noindent \textbf{Declarations:}\\

\noindent \textbf{Data availability:} Not applicable.\\
 
\noindent \textbf{Funding:} The corresponding author, Bankteshwar Tiwari, is supported by ``Incentive grant" under the IoE scheme of Banaras Hindu University, Varanasi (India) and the first author, Alok Kumar Pandey, is supported by ``CSIR Junior Research Fellowship''. S/No-132661.\\

\noindent \textbf{Competing Interests:} The authors have no relevant financial or non-financial interests to disclose.\\

\noindent \textbf{Author Contributions:} All authors contributed equally in the conceptualization and preparation of the paper. Further, all authors read and approved the final manuscript.\\


\begin{thebibliography}{1}
\bibitem{DSSZ} Bao D., Chern S. S. and Shen Z., An introduction to Riemann–Finsler geometry, Graduate Texts in Mathematics  \textbf{200}, Springer-Verlag, New York, 2000. 
\bibitem{BD} Barbilian D.,\textit{ Einordnung von Lobatschewskys Massbestimmung in gewisse allgemeine Metrik
der Jordanschen Bereiche. Casopis Praha} \textbf{64} (1935), 182-183.
\bibitem{BA} Beardon A.,\textit{ The Apollonian metric of a domain in $\mathbb{R}^n$}, P. Duren et al. (ed.), Quasiconformal mappings and analysis, Proceedings of an international symposium, Ann Arbor, 1995, Springer (1998) 91-108.
\bibitem{CXSZ} Cheng X. and Shen Z., Finsler geometry, An approach via Randers spaces, Science Press Beijing, Beijing Springer, Heidelberg, 2012.
\bibitem{SSZ} Chern S. S. and Shen Z., Riemannian-Finsler geometry, World Scientific Publisher, Singapore, 2005.
\bibitem{ZI} Ibragimov Z., \textit{On the Apollonian metrics of domain in $\mathbb{R}^n$}, Complex Variables, \textbf{48} (2003), no. 10, 837–855.
 



\bibitem{MOP}  Miyachi H., Ohshika K. and Papadopoulos A.,\textit{ Tangent spaces of the Teichmüller space of the torus with Thurston's weak metric,} Ann. Fenn. Math. \textbf{47} (2022), no. 1, 325–334.
\bibitem{AHB1}  Kumar A., Shah H.M. and Tiwari B., \textit{Isometric models of the Funk disc and the Busemann function}, Results Math. \textbf{79}, no.2, (2024).
\bibitem{AHB2} Kumar A., Shah H.M. and Tiwari B., \textit{The Funk-Finsler structure on the unit disc in the hyperbolic plane}, Mediterr. J. Math., \textbf{21} (2024).
\bibitem{OHTA} Ohta, S. I., Comparison Finsler geometry, Springer, vol. \textbf{5}, (2021).
\bibitem{PAYS} Papadopoulos A. and Yamada S., \textit{The Funk and Hilbert geometries for spaces of constant curvature}, Monatsh. Math. \textbf{172}  (2013),  97–120.
\bibitem{HHG1} Papadopoulos A. and Troyanov M., Weak Finsler structures and the Funk weak metric, Math. Proc. Camb. Phil. Soc. (2009), \textbf{147}, 419-437. 
\bibitem{HHG} Papadopoulos A. and Troyanov M., Handbook of Hilbert geometry, European Mathematical Society (EMS), Zürich, 2014.  
\bibitem{APMT} Papadopoulos A. and Troyanov M.,\textit{ Weak metrics on Euclidean domains,} JP J. Geom. Topol. \textbf{7} (2007), no. 1, 23–43.
\bibitem{RH} Ribeiro, H.,\textit{ Sur les espaces à métrique faibl}e, Portugaliae mathematica 4.1 (1943): 21-40.
\bibitem{SZ} Shen Z., Lectures on Finsler geometry, World Scientific Publishing Co., Singapore, 2001.
\bibitem{shen2} Shen Z., Differential geometry of spray and Finsler spaces, Kluwer Academic Publishers, Dordrecht, 2001.
\bibitem{Shiohama-BT}
Shiohama K. and Tiwari B., Global Riemann-Finsler Geometry, (In the memory of M. Berger), (Edited by S.G. Dani and Athanase Papadopulous) Springer, (2019), 581--621. 
\bibitem{BT} Tiwari B., A comparative overview of Riemannian and Finsler geometry, \textit{Geometry, groups and mathematical philosophy}, Contemp. Math.,  Amer. Math. Soc.   \textbf{811} (2025) 209--239.
\bibitem{SY} Yamada, S., \textit{Variational formulations of the Funk and Apollonian
weak metrics on convex sets,} RIMS K\^oky\^uroku Bessatsu, \textbf{B48} (2014), 57-72.


\end{thebibliography}
\end{document}